\documentclass[12pt,english]{amsart}
\usepackage{amsfonts,amsmath,amsthm,amscd,amssymb,latexsym,amsbsy,pb-diagram,caption,url}
\usepackage[noadjust]{cite}
\usepackage[mathscr]{eucal}
\usepackage[T2A]{fontenc}
\usepackage[cp1251]{inputenc}
\usepackage[english]{babel}
\usepackage{tikz} 
\usetikzlibrary{positioning,calc}

\newtheorem{theorem}{Theorem}[section]
\newtheorem{corollary}[theorem]{Corollary}
\newtheorem{lemma}[theorem]{Lemma}
\newtheorem{proposition}[theorem]{Proposition}

\theoremstyle{definition}
\newtheorem{definition}[theorem]{Definition}
\newtheorem{example}[theorem]{Example}

\theoremstyle{remark}
\newtheorem{remark}[theorem]{Remark}
\numberwithin{equation}{section}

\DeclareMathOperator{\diam}{diam}

\newcommand{\BB}{\mathbf{B}}
\newcommand{\RR}{\mathbb{R}}

\title{Finite Ultrametric Balls}

\author{O. Dovgoshey}
\address{Function theory department, Institute of Applied Mathematics and Mechanics of NASU, Dobrovolskogo str. 1, Slovyansk 84100, Ukraine}
\email{oleksiy.dovgoshey@gmail.com}


\begin{document}
\begin{abstract}
The necessary and sufficient conditions under which a given family \(\mathcal{F}\) of subsets of finite set \(X\) coincides with the family \(\BB_X\) of all balls generated by some ultrametric \(d\) on \(X\) are found. It is shown that the representing tree of the ultrametric space \((\BB_{X}, d_H)\) with the Hausdorff distance \(d_H\) can be obtained from the representing tree \(T_X\) of ultrametric space \((X, d)\) by adding a leaf to every internal vertex of \(T_X\).
\end{abstract}

\maketitle

\section{Introduction. Balls in ultrametric space}

The main object of research in this paper is the set of all balls in a given finite ultrametric space. The theory of ultrametric spaces is closely connected with various directions of studies in mathematics, physics, linguistics, psychology and computer science. Different properties of ultrametric spaces are described at~\cite{DM2009, DD2010, DP2013SM, Groot1956, Lemin1984FAA, Lemin1984RMS39:5, Lemin1984RMS39:1, Lemin1985SMD32:3, Lemin1988, Lemin1999, Lemin2003, QD2009, QD2014, BS2017, DM2008, DLPS2008, KS2012, Vaughan1999, Vestfrid1994, Ibragimov2012, GomoryHu(1961), PTAbAppAn2014}.

The use of trees and tree-like structures gives a natural language for description of ultrametric spaces. For the relationships between these spaces and the leaves or the ends of certain trees see~\cite{Carlsson2010, DLW, Fie, GurVyal(2012), GV, Hol, H04, BH2, Lemin2003, Ber2001, Bestvina2002}. In particular, a convenient representation of finite ultrametric spaces \((X, d)\) by monotone canonical trees was found by~V.~Gurvich and M.~Vyalyi~\cite{GurVyal(2012)}. A simple algorithm having a clear geometric interpretation was proposed in~\cite{PD(UMB)} for constructing monotone canonical trees. Following~\cite{PD(UMB)} we will say that these trees are representing trees of spaces \((X, d)\). The present paper can be considered as a development of studies initiated at~\cite{GurVyal(2012)} and continued at \cite{DDP(P-adic),DP2019, DPT(Howrigid),PD(UMB), P(TIAMM), P2018(p-Adic)}.

The paper is organized as follows.

Section~1 contains some standard definitions from the theory metric spaces and a short list of known properties of balls in ultrametric spaces.

The representing trees of finite ultrametric spaces \((X, d)\) are the main technical tool in the present paper. These trees are introduced and discussed in Section~2. In particular, in Theorem~\ref{th2.6} it is proved that the vertices of such trees are the closed balls of spaces \((X, d)\).

The main new results of the paper are concentrated in Section~3. The necessary and sufficient conditions under which a given family of subsets of finite set \(X\) coincides with the family of all balls generated by some ultrametric on \(X\) are obtained in Theorem~\ref{th1.10}. Theorem~\ref{th3.6}, Theorem~\ref{th3.11} and Proposition~\ref{p3.7} describe up to isomorphism the structure of representing trees of the metric space of all balls of finite ultrametric spaces.

Recall some definitions from the theory of metric spaces.

\begin{definition}\label{d1.1}
A \textit{metric} on a set $X$ is a function $d\colon X\times X\rightarrow \RR^{+}$, \(\RR^{+} = [0, \infty)\), such that for all $x$, $y$, $z \in X$:
\begin{enumerate}
\item $d(x,y)=d(y,x)$,
\item $(d(x,y)=0)\Leftrightarrow (x=y)$,
\item $d(x, y)\leq d(x, z) + d(z, y)$.
\end{enumerate}
\end{definition}

For every nonempty set \(A \subseteq X\), the quantity
\begin{equation}\label{eq1.3}
\diam A := \sup\{d(x, y)\colon x, y \in A\}
\end{equation}
is the \emph{diameter} of \(A\).

A metric \(d \colon X \times X \to \RR^{+}\) is an \emph{ultrametric} on $X$ if
\begin{equation}\label{eq1.2}
d(x,y) \leq \max \{d(x,z),d(z,y)\}
\end{equation}
holds for all $x$, $y$, $z \in X$. Inequality~\eqref{eq1.2} is often called the {\it strong triangle inequality}. A metric space $(X, d)$ is ultrametric if every triangle in $(X, d)$ is isosceles with base not greater than legs.

Let $(X,d)$ be a metric space. A \emph{closed ball} with a radius $r \geqslant 0$ and a center $c\in X$ is the set $B_r(c)=\{x\in X\colon d(x,c)\leqslant r\}$. The \emph{ballean} $\BB_X$ of the metric space $(X,d)$ is the set of all closed balls in $(X,d)$. Call the elements of $\BB_X$ the balls for short. Note that every one-point subset of $X$ belongs to $\BB_X$, this is a \emph{singular} ball in~$(X, d)$. For a ball $B \in \BB_X$ with a center $c$, the \emph{actual} radius of $B$ is the number
\[
\sup\{d(x, c) \colon x \in B\}.
\]

In the case of an arbitrary metric space \((X, d)\), the actual radius of a given ball \(B \in \BB_X\) can depend on which point \(c \in B\) we chose as the center of \(B\). For example, if \(X\) is the interval \([0, 2]\) with the standard metric \(d(x, y) = |x - y|\) and \(B = X\), then every \(c \in X\) can be considered as a center of \(B\) so that the corresponding actual radius equals \(\min\{c, 2-c\}\).

If $(X, d)$ is an ultrametric space and $A$ is a nonempty subset of $X$, then, using the strong triangle inequality, we can easily prove the equalities
\begin{equation}\label{eq1.1}
\diam A = \sup\{d(x,a) \colon x \in A\} = \sup\{d(x, b) \colon x \in A\}.
\end{equation}
for all \(a\), \(b\in A\). Now using~\eqref{eq1.1}, we obtain the following.

\begin{proposition}\label{pr1.2}
Let $(X, d)$ be an ultrametric space and let $B \in \BB_X$ be a ball with a radius $r$. Then the following statements hold.
\begin{enumerate}
\item\label{pr1.2:s1} Every point of $B$ is a center of $B$, i.e., the equalities
\[
B = B_r(c_1) = B_r(c_2)
\]
hold for all $c_1$, $c_2 \in B$.
\item\label{pr1.2:s2} The actual radius of $B$ is equal to the diameter of $B$ for every center of $B$.
\end{enumerate}
\end{proposition}

\begin{corollary}\label{c1.2}
Let \((X, d)\) be an ultrametric space. Then, for every \(x \in X\) and ever \(r \geqslant 0\), there is a unique ball \(B \in \BB_X\) with the radius \(r\) and such that \(x\) belongs to \(B\).
\end{corollary}

Proposition~\ref{pr1.2} and the definition of balls give us the following description of the balls in ultrametric spaces.

\begin{corollary}\label{c1.3}
Let $(X, d)$ be an ultrametric space, let $A$ be a bounded subset of $X$ and let $a \in A$. Then the following statements are equivalent:
\begin{enumerate}
\item\label{c1.3:s1} $A \in \BB_X$;
\item\label{c1.3:s2} The equivalence
\[
(d(x, a) \leqslant \diam A) \Leftrightarrow (x \in A)
\]
holds for every $x \in X$;
\item\label{c1.3:s3} The equality
\[
\BB_A = \{B \in \BB_X \colon B \subseteq A\}
\]
holds, where \(\BB_A\) is the ballean of the ultrametric space \((A, d|_{A \times A})\).
\end{enumerate}
\end{corollary}

The following proposition is well-known.

\begin{proposition}\label{p1.8}
Let $(X, d)$ be an ultrametric space and let $B_1$, $B_2 \in \BB_X$. Then the following conditions hold.
\begin{enumerate}
\item\label{p1.8:s1} $B_1 \cap B_2 \in \BB_X$ if and only if $B_1 \cap B_2 \neq \varnothing$.
\item\label{p1.8:s2} If $B_1 \cap B_2 \neq \varnothing$, then we have $B_1 \subseteq B_2$ or $B_2 \subseteq B_1$.
\item\label{p1.8:s3} $B_1 = B_2$ holds if and only if $B_1 \cap B_2 \neq \varnothing$ and
\[
\diam B_1 = \diam B_2.
\]
\end{enumerate}
\end{proposition}

Condition~\ref{p1.8:s1} of Proposition~\ref{p1.8} implies that for every bounded nonempty subset \(A\) of an ultrametric space \((X, d)\) there is the smallest ball \(B \in \BB_{X}\) containing \(A\).

The following proposition gives us a partial inversion of Proposition~\ref{pr1.2} and Proposition~\ref{p1.8}.

\begin{proposition}\label{p3.18}
Let \((X, d)\) be a finite nonempty metric space and let \(\BB_{X}\) be the set of all balls of \((X, d)\). Then the following conditions are equivalent.
\begin{enumerate}
\item \label{p3.18:s1} \(d\) is an ultrametric.
\item \label{p3.18:s2} If \(B_1\) and \(B_2\) belong to \(\BB_{X}\) and \(B_1 \cap B_2 \neq \varnothing\), then we have \(B_1 \subseteq B_2\) or \(B_2 \subseteq B_1\).
\item \label{p3.18:s3} For every ball \(B \in \BB_{X}\), each point \(b \in B\) is a center of \(B\).
\end{enumerate}
\end{proposition}

\begin{proof}
\(\ref{p3.18:s1} \Rightarrow \ref{p3.18:s3}\). This implication follows from Proposition~\ref{pr1.2}.

\(\ref{p3.18:s3} \Rightarrow \ref{p3.18:s2}\). It is evident.

\(\ref{p3.18:s2} \Rightarrow \ref{p3.18:s1}\). Suppose \ref{p3.18:s2} holds. Let \(x\), \(y\), \(z\) be points of \(X\) and \(r = d(x, y)\) and \(s = d(y, z)\). Let us consider the balls
\[
B_1 := B_r(x) \quad \text{and} \quad B_2 := B_s(z).
\]
It is clear that \(y \in B_1 \cap B_2\). According to condition~\ref{p3.18:s2} we obtain
\[
B_1 \subseteq B_2 \quad \text{or} \quad B_2 \subseteq B_1.
\]
It implies
\[
x \in B_s(z) \quad \text{or} \quad z \in B_r(x).
\]
The strong triangle inequality
\[
d(x, z) \leqslant \max\{d(x, y), d(y, z)\}
\]
follows.
\end{proof}


\begin{definition}\label{d1.4}
Let \((X, d)\) be a metric space and let \(A\) be a bounded nonempty subset of \(X\). A ball \(B^* \in \BB_X\) is the smallest ball containing \(A\) if \(B^* \supseteq A\) and the implication
\begin{equation}\label{d1.4:e1}
(B \supseteq A) \Rightarrow (B \supseteq B^*)
\end{equation}
holds for every \(B \in \BB_X\).
\end{definition}

It follows directly from~\eqref{d1.4:e1} that the smallest ball containing a given \(A\) is unique if it exists.

We can simply characterize the smallest ball containing a given bounded subset of an ultrametric space.

\begin{proposition}\label{p1.5}
Let $(X, d)$ be an ultrametric space, let $A$ be a bounded subset of $X$ and let $a \in A$. Then the ball $B_r(a)$ with $r = \diam A$ is the smallest ball containing \(A\).
\end{proposition}

\begin{corollary}\label{c1.6}
Let \((X, d)\) be an ultrametric space and let \(A\) be a bounded nonempty subset of \(X\). A ball \(B^* \in \BB_X\) is the smallest ball containing \(A\) if and only if \(B^* \supseteq A\) and the implication
\[
(B \supseteq A) \Rightarrow (\diam B \geqslant \diam B^*)
\]
holds for every \(B \in \BB_X\).
\end{corollary}

\begin{example}\label{ex1.5}
Let $(X, d)$ be an ultrametric space, $|X| \geqslant 2$, and let $a$, $b \in X$ such that
\[
\diam X = d(a, b).
\]
Then the smallest ball $B$ containing the set $\{a, b\}$ coincides with $X$.
\end{example}

There are metric spaces $(X, d)$ which are not ultrametric but for every bounded $A \subseteq X$ we can find the smallest $B \in \BB_X$ such that $A \subseteq B$.

\begin{example}\label{ex1.7}
Let $(\RR, d)$ be the metric space of all real numbers with the standard metric. The balls in $(\RR, d)$ are the finite closed intervals $[a, b]$, with $a \leqslant b$. The nonempty intersection of two bounded closed intervals in $\RR$ is also a bounded closed interval in \(\RR\). If $A \subseteq \RR$ is bounded and \(A \neq \varnothing\) and $a^* = \inf A$ and $b^* = \sup A$, then $[a^*, b^*]$ is the smallest ball containing the set $A$. Note also that \((\RR, d)\) satisfies condition~\ref{p1.8:s1} of Proposition~\ref{p1.8}.
\end{example}

\begin{figure}[htb]
\begin{tikzpicture}[scale=1]
\def\xx{0cm};
\def\dx{2cm};
\draw (\xx-\dx, 0) -- (\xx-0.25*\dx, 0);
\draw (\xx+0.25*\dx, 0) -- (\xx+\dx, 0);
\draw (\xx,0) circle (\dx);
\node at (\xx,0) [label=below:\(A\)] {};
\node at (\xx+\dx,\dx) [label=left:\(B\)] {};

\def\xx{6cm};
\draw (\xx-\dx, 0) -- (\xx-0.25*\dx, 0);
\draw (\xx+0.25*\dx, 0) -- (\xx+\dx, 0);
\draw (\xx+\dx,0) arc (-20:340:{\dx/cos(20)});
\draw (\xx+\dx,0) arc (20:380:{\dx/cos(20)});
\node at (\xx,0) [label=below:\(A\)] {};
\node at (\xx,0.6*\dx) [label=above:\(L\)] {};
\end{tikzpicture}
\caption{\(A\) is the union of two closed bounded intervals. The ball \(B\) contains \(A\) and has the minimal diameter. The lens \(L\) is the intersection of two balls containing \(A\) and it is a proper subset of \(B\).}
\label{fig1}
\end{figure}
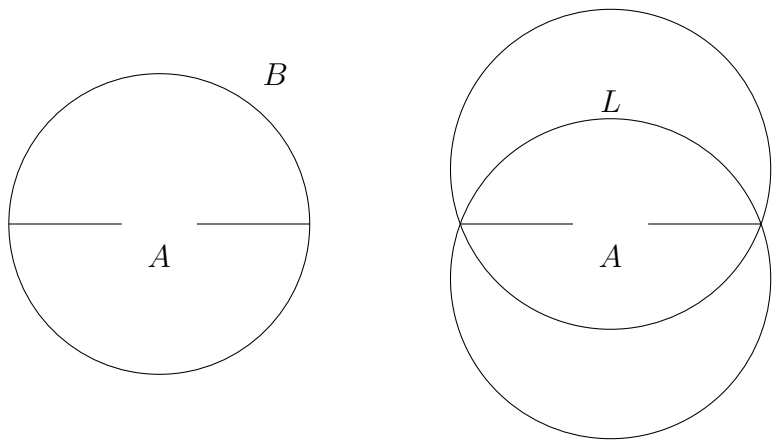

The classical Jung theorem~\cite{Jung1901} says that every nonempty bounded subset \(A\) of the \(n\)-dimensional Euclidean space \(E^n\) is contained in the unique closed ball \(B\) with the minimal diameter (see, for example, \cite{DGK1963} for the proof of Jung theorem). It should be noted that for every \(n \geqslant 2\) we can find a bonded set \(A \subseteq E^n\) such that if \(B\) is the unique ball which contains \(A\) and has the minimal diameter, then \(B\) is not the smallest ball containing \(A\) in the sense of Definition~\ref{d1.4} (see Figure~\ref{fig1}).

\section{Representing trees of finite ultrametric spaces}

Let us recall some concepts from the graph theory.

A \textit{graph} is a pair $(V, E)$ consisting of a nonempty set $V$ and a (probably empty) set $E$ of unordered pairs of different points from $V$. For a graph $G=(V,E)$, the sets $V=V(G)$ and $E=E(G)$ are called \textit{the set of vertices} and \textit{the set of edges}, respectively.

A graph \(G\) together with a function \(l \colon V(G) \to \mathbb{R}^{+}\) is called to be a \emph{labeled} graph, and \(l\) is a \emph{labeling} of \(G\).

A graph \(H\) is a \emph{subgraph} of a graph \(G\) if \(V(H) \subseteq V(G)\) and \(E(H) \subseteq E(G)\). In this case we write \(H \subseteq G\). A \emph{path} is a graph $P$ for which
$$
V(P) = \{x_0,x_1,\ldots,x_k\} \quad \text{and} \quad E(P) = \{\{x_0,x_1\},\ldots,\{x_{k-1},x_k\}\},
$$
where all $x_i$ are distinct. A finite graph $C$ is a \textit{cycle} if $|V(C)|\geq 3$ and there exists an enumeration $(v_1, v_2, \ldots, v_n)$ of its vertices such that
\begin{equation*}
(\{v_i,v_j\}\in E(C))\Leftrightarrow (|i-j|=1\quad \mbox{or}\quad |i-j|=n-1).
\end{equation*}

A connected graph without cycles is called a \emph{tree}. A vertex \(v\) of a tree $T$ is a \emph{leaf} if the \emph{degree} \(\delta(v)\) of $v$ is less than two,
\[
\delta(v) = |\{u \in V(T) \colon \{u, v\} \in E(T)\}| < 2.
\]
If a vertex $v \in V(T)$ is not a leaf of $T$, then we say that $v$ is an \emph{internal node} of $T$. A tree $T$ may have a distinguished vertex $r$ called the \emph{root}; in this case $T$ is a \emph{rooted tree} and we write $T=T(r)$. If $v$ is a vertex of a rooted tree $T = T(r)$ such that $v \neq r$, then there is a unique path $P_v \subseteq T$ which joins $v$ and $r$. A vertex $u$ of $T$ is a \emph{predecessor} of $v$ if $v \neq u$ and $u \in V(P_v)$. In this case we say that $v$ is a \emph{successor} of $u$. In particular, $v$ is a \emph{direct successor} of $u$ if $u$ and $v$ are adjacent, $\{u, v\} \in E(T)$, and $v$ is a successor of $u$.

\begin{definition}\label{d2.1}
Let $k\geqslant 2$. A graph $G$ is called \emph{complete $k$-partite} if its vertices can be divided into $k$ disjoint nonempty sets $X_1,\ldots,X_k$ so that there are no edges joining the vertices of the same set $X_i$ and every two vertices from distinct \(X_i\) and \(X_j\), $1\leqslant i, j \leqslant k$ are adjacent. In this case we write $G=G[X_1,\ldots,X_k]$.
\end{definition}
We shall say that $G$ is a {\it complete multipartite graph} if there exists $k \geqslant 2$ such that $G$ is complete $k$-partite, cf. \cite{Di}.

\begin{definition}[\cite{DDP(P-adic)}]\label{d2.2}
Let $(X,d)$ be a finite ultrametric space with \(|X| \geqslant 2\). Define a graph $G_{D, X}$ as $V(G_{D, X})=X$ and
$$
(\{u,v\}\in E(G_{D, X}))\Leftrightarrow(d(u,v)=\diam X)
$$
for all \(u\), \(v \in V(G_{D, X})\). We call $G_{D, X}$ the \emph{diametrical graph} of $X$.
\end{definition}

\begin{theorem}[\cite{DDP(P-adic)}]\label{th2.3}
Let $(X, d)$ be a finite ultrametric space, $|X|\geqslant 2$. Then $G_{D, X}$ is complete multipartite.
\end{theorem}

With every finite nonempty ultrametric space $(X, d)$, we can associate a labeled rooted tree $T_X$ by the following rule (see~\cite{PD(UMB)}). The root of \(T_X\) is the set \(X\). If $X$ is a one-point set, then $T_X$ is the rooted tree consisting of one node \(X\) with the label \(0\). Let $|X| \geqslant 2$. According to Theorem~\ref{th2.3} we have $G_{D, X} = G_{D, X}[X_1, \ldots, X_k]$. In this case the root of the tree $T_X$ is labeled by $l(X) = \diam X$ and, moreover, $T_X$ has the vertices $X_1, \ldots, X_k$, \(k \geqslant 2\), of the first level with the labels
\begin{equation}\label{eq2.1}
l(X_i) = \diam X_i,
\end{equation}
$i = 1,\ldots,k$. The vertices of the first level labeled by \(0\) are leaves, and those labeled by $\diam X_i > 0$ are internal vertices of the tree $T_X$. If the first level has no internal vertices, then the tree $T_X$ is constructed. Otherwise, by repeating the above-described procedure with $X_i$ corresponding to the internal vertices of the first level, we obtain the vertices of the second level, etc. Since $X$ is a finite set, all vertices on some level will be leaves, and the construction of $T_X$ is completed.

We shall say that the labeled rooted tree \(T_X\) is the \emph{representing tree} of \((X, d)\). The following lemma shows how we can find the distance between points of \(X\) using the labeling \(l \colon V(T_X) \to \mathbb{R}^{+}\).

\begin{lemma}\label{l2.4}
Let $(X, d)$ be a finite ultrametric space and let $x_1$ and $x_2$ be two distinct points of \(X\). If $(\{x_1\}, v_1, \ldots, v_n, \{x_2\})$ is the path joining the leaves $\{x_1\}$ and $\{x_2\}$ in $T_X$, then
\begin{equation}\label{l2.4:e1}
d(x_1, x_2) = \max\limits_{1\leqslant i \leqslant n} l({v}_i).
\end{equation}
\end{lemma}

This lemma can be found in~\cite[Lemma~1.6]{DP2018}, \cite[Lemma~2.3]{DPT(Howrigid)},  \cite[Lemma~6]{P(TIAMM)}. The proof of the lemma is completely similar to the proof of Lemma~3.2 from~\cite{PD(UMB)}. Note only that for each tree \(T\) and every pair of distinct \(u\), \(v \in V(T)\) there is a single path joining \(u\) and \(v\) in \(T\).

Let $T = T(r)$ be a rooted tree. For every vertex $v$ of $T$ we denote by $T_v$ the induced subtree of $T$ such that \(v\) is the root of \(T_v\) and
\begin{equation}\label{eq2.3}
V(T_v) = \{u \in V(T) \colon u = v \text{ or \(u\) is a successor of } v\}.
\end{equation}
In particular, we have \(T(r) = T_r\) for the case when \(v = r\). We also use the denotation \(\overline{L}(T_v)\) for the set of all leaves of \(T_v\). If \(T = T_X\) is the representing tree of a finite ultrametric space \((X, d)\) and \(v \in V(T_X)\) and
\[
\overline{L}(T_v) = \{\{x_1\}, \ldots, \{x_n\}\},
\]
then, for simplicity we write
\[
L(T_v) = \{x_1, \ldots, x_n\}.
\]
Consequently, the equality
\begin{equation}\label{eq2.4}
v = L(T_v)
\end{equation}
holds for every \(v \in V(T_X)\).

For $T_X$ consisting of one node only, \(V(T_X) = X\), we consider that \(X\) is the root of $T_X$ as well as a leaf of \(T_X\). Thus if \(X = \{x\}\), then we have \(\overline{L}(T_X) = \{\{x\}\}\) and \(L(T_X) = \{x\} = X\).

Let \(A\) and \(B\) be two nonempty bounded subsets of a metric space \((X, d)\). The \emph{Hausdorff distance} \(d_H(A, B)\) between \(A\) and \(B\) is defined by
\begin{equation}\label{eq1.4}
d_H(A, B) := \max\{\sup_{a \in A} \inf_{b \in B} d(a, b), \sup_{b \in B} \inf_{a \in A} d(a, b)\}.
\end{equation}
The definition and some properties of the Hausdorff distance can be found in~\cite{BBI2001}. We note only that if \(\{a\}\) and \(\{b\}\) are singular balls in \((X, d)\), then~\eqref{eq1.4} implies
\begin{equation}\label{eq1.5}
d_H(\{a\}, \{b\}) = d(a, b).
\end{equation}

In the following and related to it Corollary~\ref{c2.8} we put together several propositions from~\cite{P(TIAMM), DPT(Howrigid), DP2019, DPT2015}.

\begin{theorem}\label{th2.6}
Let $(X,d)$ be a finite nonempty ultrametric space with the representing tree $T_X$ and let $B_1$ and $B_2$ be distinct balls in $(X, d)$. Then the following statements are valid.
\begin{enumerate}
\item\label{th2.6:s1} The equality
\begin{equation}\label{th2.6:e0}
V(T_X) = \BB_X
\end{equation}
holds.
\item\label{th2.6:s2} $B_1$ is a direct successor of $B_2$ in \(T_X\) if and only if $B_1 \subset B_2$ and the implication
\begin{equation}\label{th2.6:e1}
(B_1 \subseteq B \subseteq B_2) \Rightarrow (B_1 = B \ \text{or} \  B_2 = B)
\end{equation}
holds for every $B \in \BB_X$.
\item\label{th2.6:s3} If \(P\) is the path joining \(B_1\) and \(B_2\) in \(T_X\), then
\begin{equation}\label{th2.6:e4}
d_H(B_1, B_2) = \max_{u \in V(P)} l(u).
\end{equation}
\end{enumerate}
\end{theorem}

For the proof of statement~\ref{th2.6:s1} and statement~\ref{th2.6:s2} see~\cite[Theorem~1.6]{DP2019}

To prove statement~\ref{th2.6:s3} we need the following lemma which claims that the Hausdorff distance between any two distinct balls \(B_1\) and \(B_2\) of a finite ultrametric spaces coincides with the diameter of the smallest ball \(B\) containing \(B_1\) and \(B_2\).

\begin{lemma}\label{l2.7}
Let \((X, d)\) be a finite nonempty ultrametric space. Then the equality
\begin{equation}\label{l2.7:e1}
d_H(B_1, B_2) = \diam(B_1 \cup B_2)
\end{equation}
holds for all distinct balls \(B_1\), \(B_2 \in \BB_X\).
\end{lemma}

\begin{proof}
Let \(B_1\), \(B_2 \in \BB_X\) and let \(B_1 \neq B_2\). It follows directly from~\eqref{eq1.4} and~\eqref{eq1.3} that the inequality
\[
d_H(B_1, B_2) \leqslant \diam(B_1 \cup B_2)
\]
holds. Consequently, it suffices to show
\begin{equation}\label{l2.7:e2}
d_H(B_1, B_2) \geqslant \diam(B_1 \cup B_2).
\end{equation}
Suppose first that there is a point \(b \in B_1\cap B_2\). Condition~\ref{p1.8:s2} of Proposition~\ref{p1.8} implies
\[
B_1 \subset B_2 \quad \text{or} \quad B_2 \subset B_1.
\]
Without loss of generality, we assume that \(B_1 \subset B_2\). Then
\begin{equation}\label{l2.7:e3}
\diam (B_1 \cup B_2) = \diam B_2
\end{equation}
and, by statement~\ref{p1.8:s3} of Proposition~\ref{p1.8}, the inequality
\[
\diam B_2 > \diam B_1
\]
holds.

By statement~\ref{th2.6:s1} of Theorem~\ref{th2.6} we have \(B_1\), \(B_2 \in V(T_X)\), where \(T_X\) is the representing tree of \((X, d)\).

Let \(G_{D, B_2} = G_{D, B_2}[B_2^1, \ldots, B_2^m]\) be the diametrical graph of \(B_2\). Using statement~~\ref{th2.6:s2} of Theorem~\ref{th2.6} and the inclusion \(B_1 \subset B_2\) we can find \(j \in \{1, \ldots, m\}\) such that
\[
B_1 \subseteq B_2^j.
\]
Without loss of generality we set
\begin{equation}\label{l2.7:e4}
B_1 \subseteq B_2^1.
\end{equation}
It follows from~\eqref{eq1.4}, \eqref{l2.7:e4}, Theorem~\ref{th2.3} and the equality
\[
B_2 = \bigcup_{j=1}^{m} B_2^{j},
\]
that
\begin{equation}\label{l2.7:e5}
d_H(B_1, B_2) \geqslant \sup_{a \in B_2} \inf_{b \in B_1} d(a, b) \\
\geqslant \sup_{a \in B_2^2} \inf_{b \in B_2^1} d(a, b) = \diam B_2.
\end{equation}
Now~\eqref{l2.7:e2} follows from~\eqref{l2.7:e5} and~\eqref{l2.7:e3}. The case when \(B_1\) and \(B_2\) are disjoint can be considered similarly.
\end{proof}

\begin{remark}\label{r2.7}
Lemma~\ref{l2.7} can also be obtained from Lemma~2.2 of~\cite{QD2014} and it holds for arbitrary ultrametric space \((X, d)\) but in the present paper we consider only finite \((X, d)\).
\end{remark}

\begin{proof}[Proof of statement~\ref{th2.6:s3} from Theorem~\(\ref{th2.6}\).] Let \(P\) be the unique path joining \(B_1\) and \(B_2\) in \(T_X\). By Lemma~\ref{l2.7}, we have the equality
\begin{equation}\label{th2.6:e5}
d_H(B_1, B_2) = \diam (B_1 \cup B_2).
\end{equation}
Using Proposition~\ref{p1.5}, we can find the smallest ball \(B \in \BB_X\) such that
\begin{equation}\label{th2.6:e6}
B \supseteq B_1 \cup B_2.
\end{equation}
The ball \(B\) is unique and the equality
\begin{equation}\label{th2.6:e7}
\diam B = \diam (B_1 \cup B_2)
\end{equation}
holds. Inclusion~\eqref{th2.6:e6} holds if and only if
\[
B \supseteq B_1 \quad \text{and} \quad B \supseteq B_2.
\]
Let \(P\) be the path joining \(B_1\) and \(B_2\) in \(T_X\) and let \(P_i\) be the path joining \(B_i\) and \(B\) in \(T_X\), \(i = 1\), \(2\). From the already proven statement~\ref{th2.6:s2} and the uniqueness of the paths \(P_1\), \(P_2\) and \(P\) it follows that
\[
P \subseteq P_1 \cup P_2,
\]
i.e., we have
\begin{equation}\label{th2.6:e8}
V(P) \subseteq V(P_1) \cup V(P_2) \quad \text{and} \quad E(P) \subseteq E(P_1) \cup E(P_2).
\end{equation}
If \(u \in V(P_i)\), \(i = 1\), \(2\), and \(u \neq B\), then \(u\) is a successor of \(B\), that implies the inequality
\[
l(u) < l(B).
\]
Equality~\eqref{th2.6:e4} follows from the last inequality and \eqref{th2.6:e5}, \eqref{th2.6:e7}, \eqref{th2.6:e8}.
\end{proof}

It is well-known that the Hausdorff distance \(d_H\) is a metric on the space of all bounded, closed, nonempty subsets of arbitrary metric space \((X, d)\) (see, for example, \cite[Proposition~7.3.3]{BBI2001} ).

\begin{corollary}\label{c2.8}
Let \((X, d)\) be a finite nonempty ultrametric space, let \(\BB_X\) be the ballean of \((X, d)\) and let \(d_H\) be the restriction of the Hausdorff distance on \(\BB_X\). Then the metric space \((\BB_X, d_H)\) is ultrametric.
\end{corollary}

\begin{proof}
It suffices to show that the strong triangle inequality
\begin{equation}\label{c2.8:e1}
d_H(B_1, B_2) \leqslant \max\{d_H(B_1, B_3), d_H(B_2, B_3)\}
\end{equation}
holds for all \(B_1\), \(B_2\), \(B_3 \in \BB_X\). Inequality~\eqref{c2.8:e1} is trivial if \(B_1 = B_2\) or \(B_1 = B_3\) or \(B_2 = B_3\). Suppose \(B_1\), \(B_2\), \(B_3\) are distinct. By statement~\ref{th2.6:s1} of Theorem~\ref{th2.6} we have \(B_1\), \(B_2\), \(B_3 \in V(T_X)\), where \(T_X\) is the representing tree of \((X, d)\). Let \(P\) be the unique path joining \(B_1\) and \(B_2\) in \(T_X\) and let \(P_i\), \(i = 1\), \(2\), be the unique path joining \(B_i\) and \(B_3\) in \(T_X\). Then we have
\begin{equation}\label{c2.8:e2}
V(P) \subseteq V(P_1) \cup V(P_2).
\end{equation}
Now inequality~\eqref{c2.8:e1} follows from~\eqref{c2.8:e2} and statement~\ref{th2.6:s3} of Theorem~\ref{th2.6}.
\end{proof}

\begin{remark}\label{r2.9}
Corollary~\ref{c2.8} remains valid for arbitrary nonempty ultrametric space \((X, d)\), that follows from Lemma~2.4 of \cite{QD2014}.
\end{remark}

\begin{remark}\label{r2.5}
Considering the minimal spanning tree for the compete weighted graph generated by ultrametric space \((\mathbf{B}_X, d_{H})\) (see, for example, Theorem~1 in paper \cite{GurVyal(2012)}) we obtain a natural dual form of Lemma~\ref{l2.4}.
\end{remark}

Since, for a finite nonempty ultrametric space \((X, d)\), the metric space \((\BB_X, d_H)\) is also finite and ultrametric, we can construct the representing tree \(T_{\BB_X}\). In the next section of the paper, starting from the representing tree \(T_X\), we give a simple inductive rule for constructing all the trees belonging to the sequence \(T_X\) \(T_{\BB_X}\), \(T_{\BB_{\BB_X}}\), \(T_{\BB_{\BB_{\BB_X}}}\), \(\ldots\).

\begin{lemma}\label{l2.9}
Let \((X, d)\) be a finite ultrametric space and let
\[
G_{D, X} = G_{D, X}[X_1, \ldots, X_k], \quad k \geqslant 2,
\]
be the diametrical graph of \((X, d)\). Then the diametrical graph \(G_{D, \BB_X}\) of the ultrametric space \((\BB_X, d_H)\) is the complete \((k+1)\)-partite graph,
\begin{equation}\label{l2.9:e1}
G_{D, \BB_X} = G_{D, \BB_X}[\{X\}, \BB_{X_1}, \ldots, \BB_{X_k}]
\end{equation}
where \(\BB_{X_1}\), \(\ldots\), \(\BB_{X_k}\) are the balleans of the ultrametric spaces \((X_1, d|_{X_1 \times X_1})\), \(\ldots\), \((X_1, d|_{X_k \times X_k})\).
\end{lemma}

\begin{proof}
Let \(Y = \BB_{X}\), \(Y_0 = \{X\}\) and \(Y_i = \BB_{X_i}\) for \(i = 1\), \(\ldots\), \(k\). According to formula~\eqref{eq1.5} and Theorem~\ref{th2.6} (formula~\eqref{th2.6:e4}) we have
\[
l(Y) = \diam Y = \diam \BB_{X} = l(X) = \diam X,
\]
and
\[
\diam Y_i = \diam \BB_{X_i} = \diam X_i < \diam X,
\]
and \(d_H(\{y_i\}, Y_0) = \diam X\) for every \(i \in \{1, \ldots, k\}\), and
\[
d(y_i, y_j) = d_H(\{y_i\}, \{y_j\}) = \diam X
\]
for all \(y_i \in Y_i\), \(y_j \in Y_j\) with distinct \(i\), \(j \in \{1, \ldots, k\}\).

Statements~\ref{th2.6:s1} and \ref{th2.6:s2} of Theorem~\ref{th2.6} and Corollary~\ref{c1.3} imply
\[
Y = \bigcup_{i = 0}^{k} Y_i.
\]
It remains to use Definition~\ref{d2.1} and Definition~\ref{d2.2}
\end{proof}

\section{Characterizations of balleans of finite ultrametric spaces up to isomorphisms}

First of all we recall the definition of the isomorphism of graphs.

\begin{definition}\label{d3.1}
Let $G_1$ and $G_2$ be finite graphs. A bijection $f\colon V(G_1)\to V(G_2)$ is an isomorphism of $G_1$ and $G_2$ if
\begin{equation}\label{d3.1e1}
(\{u,v\} \in E(G_1)) \Leftrightarrow (\{f(u),f(v)\} \in E(G_2))
\end{equation}
holds for all $u$, $v \in V(G_1)$. The graphs $G_1$ and $G_2$ are isomorphic if there exists an isomorphism $f\colon V(G_1) \to V(G_2)$.

If $G_1 = G_1(r_1)$ and $G_2 = G_2(r_2)$ are rooted graphs, then $G_1$ and $G_2$ are isomorphic (as rooted graphs) if there is an isomorphism $f\colon V(G_1) \to V(G_2)$ such that $f(r_1) = r_2$.
\end{definition}

Note that every isomorphism \(f \colon V(T_1) \to V(T_2)\) of the rooted trees \(T_1 = T_1(r_1)\) and \(T_2 = T_2(r_2)\) preserves the orientations corresponding to the choice of the roots \(r_1\) and \(r_2\) (see~\cite{HN} for the definition and properties of homomorphisms of directed graphs).

The next main definition is the definition of isomorphic labeled rooted trees.

\begin{definition}\label{d3.2}
Let \(T_i=T_i(r_i,l_i)\), \(i=1\), \(2\), be labeled rooted trees with the roots $r_i$ and the labelings \(l_i\colon V(T_i)\to \RR^{+}\). An isomorphism \(f\colon V(T_1) \to V(T_2)\) of the rooted trees $T_1(r_1)$ and $T_2(r_2)$ is an isomorphism of the labeled rooted trees $T_1(r_1,l_1)$ and $T_2(r_2,l_2)$ if the equality \begin{equation}\label{d3.2e1}
l_2(f(v))=l_1(v)
\end{equation}
holds for every $v \in V(T_1)$. The labeled rooted trees $T_1(r_1,l_1)$ and $T_2(r_2,l_2)$ are isomorphic if there is an isomorphism of these trees.
\end{definition}

Recall that two metric spaces $(X,d)$ and $(Y, \rho)$ are \emph{isometric} if there is a bijection $f\colon X\to Y$ such that the equality
$$
d(x,y)=\rho(f(x),f(y))
$$
holds for all $x$, $y \in X$.

\begin{theorem}[{\cite{DPT(Howrigid)}}]\label{l3.3}
Let $(X,d)$ and $(Y, \rho)$ be nonempty finite ultrametric spaces. The representing trees $T_X$ and $T_Y$ are isomorphic as labeled rooted trees if and only if $(X,d)$ and $(Y, \rho)$ are isometric.
\end{theorem}

Let \(T = T(r)\) be a rooted tree. Then, for every \(v \in V(T)\), we denote by \(\delta^{+}(v)\) the number of direct successors of \(v\).

The following theorem and their corollary can be found in~\cite{DP2019} but in view of the importance of this result in the context of the paper, we present it with a short proof.

\begin{theorem}[\cite{DP2019}]\label{l3.4}
Let $T=T(r, l_T)$ be a finite labeled rooted tree with the root $r$ and the labeling $l_T \colon V(T) \to \mathbb{R}^+$. Then the following two conditions are equivalent.
\begin{enumerate}
\item\label{l3.4:s1} For every $u \in V(T)$ we have $\delta^+(u)\neq 1$ and
$$
(\delta^+(u) =0) \Leftrightarrow (l_T(u)=0)
$$
and, in addition, the inequality
\begin{equation}\label{l3.4e1}
l_T(v) < l_T(u)
\end{equation}
holds whenever $v$ is a direct successor of $u$.
\item\label{l3.4:s2} There is a unique up to isometry finite ultrametric space $(X,d)$ such that the representing tree $T_X$ and $T$ are isomorphic as labeled rooted trees.
\end{enumerate}
\end{theorem}

\begin{proof}
$\ref{l3.4:s1} \Rightarrow \ref{l3.4:s2}$. Let us denote by $X$ the set of leaves of $T$. For every pair of distinct $x$, $y \in X$ we denote by $S_{x,y}$ the subset of $V(T)$ consisting of all vertices $w$ for which $x$ and $y$ are successors of $w$ and define a function $d \colon X \times X \to \mathbb{R}^+$ as
\begin{equation}\label{l3.4e2}
d(x,y) := \begin{cases}
0, & \text{if $x = y$}\\
\min_{w \in S_{x,y}} l_T(w), & \text{if $x \neq y$}.
\end{cases}
\end{equation}
Using condition \ref{l3.4:s1}, we can prove that $d$ is an ultrametric on $X$. Now~\eqref{l3.4e2}, the definition of the representing trees and Lemma~\ref{l2.4} imply that $T_X$ and $T(r, l_T)$ are isomorphic as labeled rooted trees. The uniqueness of \((X, d)\) for which \(T_X\) and \(T\) are isomorphic follows from Theorem~\ref{l3.3}.

$\ref{l3.4:s2} \Rightarrow \ref{l3.4:s1}$. If $(X,d)$ is a finite nonempty ultrametric space, then condition~$(i)$ evidently holds for $T=T_X$. Moreover, if we have two isomorphic labeled rooted trees and one of them satisfies condition~\ref{l3.4:s1}, then another also satisfies~\ref{l3.4:s1}. The implication $\ref{l3.4:s2} \Rightarrow \ref{l3.4:s1}$ follows.
\end{proof}

Using the proofs of Theorem~\ref{l3.4} and Theorem~\ref{th2.6}, we obtain the following.

\begin{corollary}\label{c3.4}
Let \(T = T(r, l_T)\) be a finite labeled rooted tree satisfying condition~\ref{l3.4:s1} of Theorem~\ref{l3.4}, let \(X\) be the set of leaves of \(T\) and let \((X, d)\) be the ultrametric space with \(d \colon X \times X \to \RR^{+}\) defined by~\eqref{l3.4e2}. For every internal vertex \(v\) of \(T\) we denote by \(\overline{L}(T_v)\) the set of all leaves of \(T\) which are successors of \(v\). Then the mapping \(\Phi \colon V(T) \to V(T_X)\), where \(T_X = T_X(X, l)\) is the representing tree of \((X, d)\) and
\begin{equation}\label{c3.4:e1}
\Phi(v) = \begin{cases}
\overline{L}(T_v), & \text{if \(v\) is an internal vertex of \(T_X\)}\\
\{v\}, & \text{if \(v\) is a leaf of \(T\)},
\end{cases}
\end{equation}
is an isomorphism of the labeled rooted trees \(T\) and \(T_X\).
\end{corollary}

\begin{remark}\label{r3.6}
Using statement~\ref{th2.6:s1} of Theorem~\ref{th2.6}, we see that the image of the mapping \(\Phi\) defined by~\eqref{c3.4:e1} is the ballean of the ultrametric space \((X, d)\).
\end{remark}

Theorem~\ref{l3.4} also implies the following important corollary.

\begin{corollary}[\cite{DP2019}]\label{c3.5}
Let $T=T(r)$ be a finite rooted tree. Then the following conditions are equivalent.
\begin{enumerate}
\item[$(i)$] For every $u \in V(T)$ we have $\delta^+(u)\neq 1$.
\item[$(ii)$] There is a finite ultrametric space $(X,d)$ such that $T_X$ and $T$ are isomorphic as rooted trees.
\end{enumerate}
\end{corollary}

\begin{definition}\label{d3.6}
Let \(T_1\) and \(T_2\) be trees and let \(x\) be a leaf of \(T_2\). Suppose we have
\[
V(T_2) = V(T_1) \cup \{x\}, \quad x \notin V(T_1), \quad E(T_1) \subseteq E(T_2).
\]
Then there is a unique vertex \(u \in V(T_1)\), such that \(\{x, u\} \in E(T_2)\). In this case we say that \(T_2\) is obtained by \emph{adding the leaf} \(x\) \emph{to the vertex} \(u\).
\end{definition}

In the following theorem we consider an ultrametric space \((X, d)\) with \(X\) satisfying the condition
\begin{equation}\label{eq3.5}
\{Y\} \not\subseteq X
\end{equation}
for every \(Y \subseteq X\). In accordance with Definition~\ref{d3.6}, this condition allows us to add leaf \(\{u\}\) to the internal vertex \(u\) of the tree \(T_X\) (see equality~\eqref{eq2.4}). We note that for every ultrametric space \((Z, \rho)\) there is an ultrametric space \((X, d)\) such that \((Z, \rho)\) and \((X, d)\) are isometric and, in addition, \eqref{eq3.5} holds for every \(Y \subseteq X\).

\begin{theorem}\label{th3.6}
Let \((X, d)\) be a finite ultrametric space with the representing tree \(T_X\), let \(\BB_{X}\) be the ballean of \((X, d)\) and let \(d_H\) be the Hausdorff distance on \(\BB_{X}\). Then the representing tree \(T_{\BB_{X}}\) of the ultrametric space \((\BB_{X}, d_H)\) and the rooted tree obtained from \(T_X\) by adding a leaf to every internal vertex of \(T_X\) are isomorphic as rooted trees.
\end{theorem}

This theorem can be obtained as a simple consequence of Lemma~\ref{l2.9} and description of representing trees given after Theorem~\ref{th2.3} but in view of the importance of Theorem~\ref{th3.6} for what follows we give a detailed proof below.

\begin{proof}[Proof of Theorem~\ref{th3.6}.]
Suppose first that \(X\) is a single-point set, \(X = \{x\}\). Then we have \(V(T_X) = \{x\}\) and \(V(T_{\BB_{X}}) = \{X\} = \{\{x\}\}\). The theorem holds because \(T_X\) does not contain any internal vertices. In this case, \(T_{\BB_{X}}\) is the rooted tree consisting of one vertex \(\{X\}\) with the label \(0\). Let \(|X| \geqslant 2\). According to Lemma~\ref{l2.9} and Theorem~\ref{th2.3} there is an integer number \(k \geqslant 2\) such that
\[
G_{D, X} = G_{D, X}[X_1, \ldots, X_k] \text{ and } G_{D, \BB_{X}} = G_{D, \BB_{X}}[\{X\}, \BB_{X_1}, \ldots, \BB_{X_k}].
\]
Then \(X_1\), \(\ldots\), \(X_k\) are the vertices of the first level of \(T_X\) and \(\{X\}\), \(\BB_{X_1}\), \(\ldots\), \(\BB_{X_k}\) are the vertices of the first level of \(T_{\BB_{X}}\). Note that \(\{X\}\) is an one-point set consisting from the ball \(\{X\} \in \BB_{X}\). Consequently, we have
\[
\diam \{X\} = 0,
\]
so that \(\{X\}\) is a leaf of \(T_{\BB_{X}}\). Write \(T_1\) for the rooted tree obtained from \(T_X\) by adding the leaf \(\{X\}\) to the root \(X\) of \(T_X\). Since all vertices of \(T_X\) are subsets of \(X\), condition~\eqref{eq3.5} and Definition~\ref{d3.6} imply that \(T_1\) is correctly defined.

The vertices of \(T_{\BB_{X}}\) of first level labeled by \(0\) are leaves of \(T_{\BB_{X}}\) and those labeled by \(\diam \BB_{X_i} > 0\) are internal vertices.

If the first level has no internal vertices, then the tree \(T_{\BB_{X}}\) is constructed. Otherwise, by applying the above described procedure to \(\BB_{X_i}\) with \(\diam \BB_{X_{i}} > 0\), we construct the vertices of the second level of \(T_{\BB_{X}}\) and a rooted tree \(T_2\) obtained from \(T_1\) by adding the leaf \(\{X_i\}\) to every internal vertex \(X_i\) of the first level in \(T_1\). Using~\eqref{eq3.5}, we see that \(T_2\) is also correctly defined. Since \(X\) is finite, all vertices on some level \(n\) will be leaves. Thus we will construct the rooted tree \(T_{\BB_{X}}\) and a rooted tree \(T_n\) such that:
\begin{itemize}
\item \(V(T_n) = V(T_X) \cup \{\{B\} \colon B \in \BB_{X} \text{ and } \diam B > 0\}\);
\item \(T_n\) is obtained from \(T_X\) by adding the leaf \(\{B\}\) to every nonsingular ball (\(= \text{internal vertex}\)) \(B\) of \(T_X\).
\end{itemize}

To complete the proof it suffices to note that the function
\[
f \colon V(T_n) \to V(T_{\BB_{X}})
\]
with
\[
f(u) = \begin{cases}
\BB_{u} & \text{if \(u\) is a vertex of \(T_X\) (= is a ball in \((X, d)\))}\\
u & \text{if } u \notin V(T_X)
\end{cases}
\]
is an isomorphism of \(T_n\) and \(T_{\BB_{X}}\).
\end{proof}

Theorem~\ref{th3.6} and Corollary~\ref{c3.5} give us the following.

\begin{proposition}\label{p3.7}
Let \(T = T(r)\) be a finite rooted tree and let \(Ch(v)\) denote the set of all direct successors of \(v\) for every \(v \in V(T)\). Then the following conditions are equivalent.
\begin{enumerate}
\item\label{p3.7:s1} For every \(u \in V(T)\) we have
\[
\delta^{+}(u) \notin \{1, 2\}
\]
and \(Ch(v)\) contains a leaf for every internal vertex \(v \in V(T)\).
\item\label{p3.7:s2} There is a finite nonempty ultrametric space \((X, d)\) such that \(T_{\BB_{X}}\) and \(T\) are isomorphic as rooted trees.
\end{enumerate}
\end{proposition}

\begin{example}\label{ex3.8}
Let \((X, d)\) be an ultrametric space such that \(X = \{x_1, x_2, x_3, x_4\}\) and
\[
d(x_1, x_2) = d(x_2, x_3) = d(x_3, x_4) = d(x_4, x_1) = \diam X
\]
and
\[
d(x_1, x_3) = d(x_2, x_4) < \diam X.
\]
Then \(\delta^{+}(u_0) = 2\) holds for every internal vertex \(u_0 \in V(T_X)\), and \(\delta^{+}(u_1) = 3\) holds for every internal vertex \(u_1 \in V(T_{\BB_{X}})\), and \(\delta^{+}(u_2) = 4\) holds for every internal vertex \(u_2 \in V(T_{\BB_{\BB_X}})\) and so on \(\ldots\) (see Figure~\ref{fig2}).
\end{example}

\begin{figure}[htb]
\begin{tikzpicture}[scale=1]
\tikzstyle{level 1}=[level distance=10mm,sibling distance=1.2cm]
\tikzstyle{level 2}=[level distance=10mm,sibling distance=0.5cm]
\tikzstyle{level 3}=[level distance=10mm,sibling distance=.5cm]
\tikzset{solid node/.style={circle,draw,inner sep=1.5pt,fill=black}}

\node at (0,4.5) [label=right:\(T_X\)] {};
\node at (0,4) [solid node] {}
child{node [solid node] {}
	child{node [solid node] {}}
	child{node [solid node] {}}
}
child{node [solid node] {}
	child{node [solid node] {}}
	child{node [solid node] {}}
};

\node at (6,4.5) [label=right:\(T_{\BB_{X}}\)] {};
\node at (6,4) [solid node, sibling distance=1cm] {}
child{node [solid node] {}
	child{node [solid node] {}}
	child{node [solid node] {}}
	child{node [solid node] {}}
}
child{node [solid node] {}}
child{node [solid node] {}
	child{node [solid node] {}}
	child{node [solid node] {}}
	child{node [solid node] {}}
};

\tikzstyle{level 1}=[level distance=10mm,sibling distance=0.8cm]
\node at (0,0.5) [label=right:\(T_{\BB_{\BB_{X}}}\)] {};
\node at (0,0) [solid node] {}
child{node [solid node] {}
	child{node [solid node] {}}
	child{node [solid node] {}}
	child{node [solid node] {}}
	child{node [solid node] {}}
}
child{node [solid node] {}}
child{node [solid node] {}}
child{node [solid node] {}
	child{node [solid node] {}}
	child{node [solid node] {}}
	child{node [solid node] {}}
	child{node [solid node] {}}
};

\tikzstyle{level 1}=[level distance=10mm,sibling distance=0.8cm]
\node at (6,0.5) [label=right:\(T_{\BB_{\BB_{\BB_{X}}}}\)] {};
\node at (6,0) [solid node] {}
child{node [solid node] {}
	child{node [solid node] {}}
	child{node [solid node] {}}
	child{node [solid node] {}}
	child{node [solid node] {}}
	child{node [solid node] {}}
}
child{node [solid node] {}}
child{node [solid node] {}}
child{node [solid node] {}}
child{node [solid node] {}
	child{node [solid node] {}}
	child{node [solid node] {}}
	child{node [solid node] {}}
	child{node [solid node] {}}
	child{node [solid node] {}}
};
\end{tikzpicture}
\caption{Beginning with the second member, each member of the sequence \(T_X\), \(T_{\BB_{X}}\), \(T_{\BB_{\BB_X}}\), \(T_{\BB_{\BB_{\BB_X}}}\), \(\ldots\) is obtained from the previous one by adding a leaf to every internal vertex of this previous.}
\label{fig2}
\end{figure}
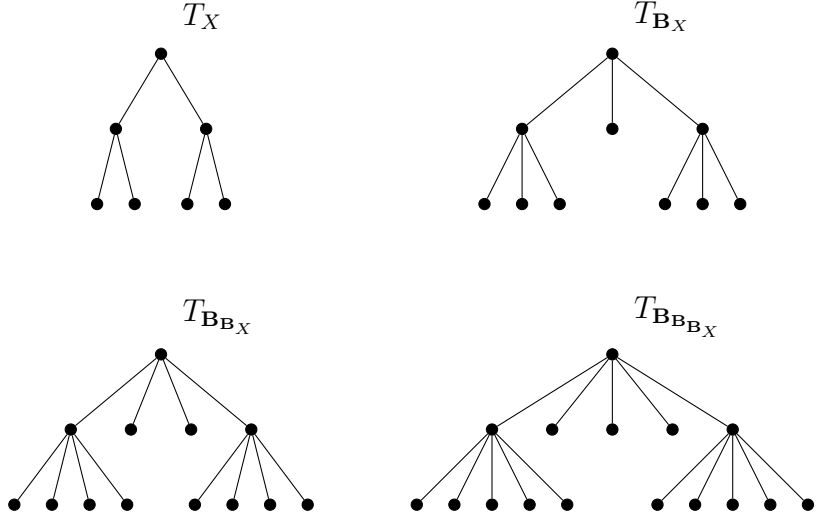

\begin{remark}\label{r3.10}
It is interesting to note that the building a tree from another one by adding a leaf to each internal vertex also occurs in some semantic studies (see~\cite{JLZ2018}).
\end{remark}

Let \((X, d)\) be a finite nonempty ultrametric space. Let us define a sequence \((\BB_{X}^{(n)})\) of finite ultrametric spaces as follows
\[
\BB_{X}^{(0)} := (X, d), \quad \BB_{X}^{(1)} := (\BB_{X}, d_H), \ldots, \quad \BB_{X}^{(n)} := (\BB_{\BB_{X}^{(n-1)}}, d_H), \ldots,
\]
i.e., for every \(n \geqslant 1\), \(\BB_{X}^{(n)}\) is the ballean of \(\BB_{X}^{(n-1)}\) together with the Hausdorff distance generated by distance on \(\BB_{X}^{(n-1)}\).

The following proposition is a corollary of Theorem~\ref{l3.3} and Theorem~\ref{th3.6}.

\begin{proposition}\label{p3.13}
Let \((X, d)\) and \((Y, \rho)\) be finite nonempty ultrametric spaces. If there is \(n \geqslant 0\) such that \((\BB_{X}^{(n)}, d_H)\) and \((\BB_{Y}^{(n)}, \rho_H)\) are isometric, then \((X, d)\) and \((Y, \rho)\) are also isomeric.
\end{proposition}

\begin{proof}
It follows from Theorem~\ref{l3.3} and Theorem~\ref{th3.6} by induction on \(n\).
\end{proof}

The following theorem is a generalization of Theorem~\ref{l3.4}.

\begin{theorem}\label{th3.11}
Let \(T = T(r, l)\) be a finite labeled rooted tree with the root \(r\) and the labeling \(l \colon V(T) \to \RR^{+}\) and let \(n \geqslant 0\) be an integer number. \(T\) is isomorphic (as a labeled rooted tree) to a representing tree \(T_{\BB_{X}^{(n)}}\) for some finite nonempty ultrametric space \((X, d)\) if and only if the following conditions hold.
\begin{enumerate}
\item \label{th3.11:s1} For every \(u \in V(T)\) we have
\[
\delta^{+}(u) \notin \{1, \ldots, n+1\}.
\]
\item \label{th3.11:s2} The inequality \(l(v) < l(u)\) holds for \(u\), \(v \in V(T)\), whenever \(v\) is direct successor of \(u\).
\item \label{th3.11:s3} The inequality
\[
|Ch(u) \cap \overline{L}(T)| \geqslant n
\]
holds for every internal vertex of \(T\), where \(Ch(u)\) and \(\overline{L}(T)\) are the set of direct successors of \(u\) and the set of leaves of \(T\), respectively.
\end{enumerate}
\end{theorem}

\begin{proof}[Sketch of the proof]
The theorem holds for \(n = 0\). In this case, it is a reformulation of Theorem~\ref{l3.4}. The general case can be obtained by induction on \(n\) with the help of Lemma~\ref{l2.9} and Theorem~\ref{th3.6}. Note only that for every finite ultrametric space \((Y, \rho)\) and every \(Z \in \BB_{Y}\) we have the equalities
\begin{multline*}
\diam Z = \sup\{\rho(x, y)\colon x, y \in Z\} \\
= \diam \BB_{Z} = \sup\{d_H(B_1, B_2) \colon B_1, B_2 \in \BB_{Z}\}
\end{multline*}
so that condition~\ref{th3.11:s2} in the present theorem is a reformulation of the corresponding condition from Theorem~\ref{l3.4}.
\end{proof}

The next theorem gives the necessary and sufficient conditions under which a given family of subsets of a finite set \(X\) coincides with the family of all balls generated by ultrametric on \(X\).

\begin{theorem}\label{th1.10}
Let \(X\) be a finite nonempty set and let \(\mathcal{F}\) be a nonempty set of nonempty subsets of \(X\). Then the following statements are equivalent.
\begin{enumerate}
\item\label{th1.10:s1} There is an ultrametric \(d \colon X \times X \to \RR^{+}\) such that \(\mathcal{F}\) is the ballean of \((X, d)\),
\[
\mathcal{F} = \BB_X.
\]
\item\label{th1.10:s2} The set \(\mathcal{F}\) satisfies the conditions:
\begin{enumerate}
\item \(X \in \mathcal{F}\) and \(\{x\} \in \mathcal{F}\) for every \(x \in X\);
\item If \(X_1\), \(X_2 \in \mathcal{F}\) and \(X_1 \cap X_2 \neq \varnothing\), then we have
\[
X_1 \subseteq X_2 \quad \text{or}\quad X_2 \subseteq X_1.
\]
\end{enumerate}
\end{enumerate}
\end{theorem}

\begin{proof}
\(\ref{th1.10:s1} \Rightarrow \ref{th1.10:s2}\). The validity of this implication follows from the definition of the balls and statement~\ref{p1.8:s2} of Proposition~\ref{p1.8}.

\(\ref{th1.10:s2} \Rightarrow \ref{th1.10:s1}\). For every pair \(x\), \(y \in X\) we write
\begin{equation}\label{th1.10:e1}
d(x, y) := \min\{|F| \colon x, y \in F \in \mathcal{F}\} - 1.
\end{equation}
If \(x\), \(y\), \(z\) are arbitrary points of \(X\) and \(\{x, z\} \subseteq F_1 \in \mathcal{F}\) and \(\{y, z\} \subseteq F_2 \in \mathcal{F}\), then we obtain
\[
d(x, y) \leqslant \max \{|F_1|, |F_2|\} - 1 = \max \{|F_1| - 1, |F_2| - 1\}.
\]
It implies the strong triangle inequality
\[
d(x, y) \leqslant \max\{d(x, z), d(z, y)\}.
\]
The function \(d\) is symmetric, moreover, from \((ii_1)\) and \eqref{th1.10:e1} it follows that \(d(x, x) = 0\) holds for all \(x \in X\). Thus \(d\) is an ultrametric on \(X\).

Let \(B\) be a closed ball in \((X, d)\). Then there are \(x \in X\) and \(r \geqslant 0\) such that \(B = B_r(x)\). It is clear that
\[
B_r(x) = \cup \{F \in \mathcal{F} \colon x \in F \text{ and } |F| \leqslant \lfloor r\rfloor\} \in \mathcal{F}
\]
for every \(x \in X\) and every \(r \geqslant 0\), where \(r \mapsto \lfloor r\rfloor\) is the floor function. Consequently, the inclusion
\[
\mathbf{B}_X \subseteq \mathcal{F}
\]
holds. It remains to prove the inclusion
\begin{equation}\label{th1.10:e2}
\mathcal{F} \subseteq \mathbf{B}_X.
\end{equation}
Let \(F\) belong to \(\mathcal{F}\) and let \(x\) belong to \(F\). We claim that
\[
F = B_r(x)
\]
holds with \(r = |F| - 1\). From~\eqref{th1.10:e1} it follows that \(F \subseteq B_r(x)\). Suppose that \(B_r(x) \setminus F \neq \varnothing\). If \(x_1\) belongs to \(B_r(x) \setminus F\), then \(d(x, x_1) \leqslant r\) and using \eqref{th1.10:e1} again we can find \(F_1 \in \mathcal{F}\) such that \(|F_1| \leqslant |F|\) and \(\{x, x_1\} \subseteq F_1\). Since \(x \in F \cap F_1\) and \(x_1 \in F_1 \setminus F\), condition~\((ii_2)\) implies \(F_1 \supseteq F\). Consequently, the inequality \(|F_1| \geqslant |F| + 1\) holds, contrary to \(|F_1| \leqslant |F|\). Inclusion~\eqref{th1.10:e2} follows.
\end{proof}

Analyzing the above proof we obtain a constructive variant of Theorem~\ref{th1.10}.

\begin{theorem}\label{th3.15}
Let \(X\) be a finite nonempty set and let \(\mathcal{F}\) be a nonempty set of nonempty subsets of \(X\). If we can find an ultrametric \(d \colon X \times X \to \RR^{+}\) such that \(\mathcal{F}\) is the ballean of \((X, d)\), then the function \(\tau \colon X\times X \to \RR^{+}\) defined as
\[
\tau(x, y) = \min\{|X_i| -1 \colon  x, y \in X_i \text{ and } X_i \in \mathcal{F}\},
\]
where \(|X_i|\) is the number of elements of \(X_i\), is also an ultrametric on \(X\) and \(\mathcal{F}\) is the ballean of the ultrametric space \((X, \tau)\).
\end{theorem}

\begin{example}\label{ex1.11}
Let \(X = \{x_1, \ldots, x_n\}\) be a finite set and let
\[
\mathcal{F} = \{X, \{x_1\}, \ldots, \{x_n\}\}.
\]
Then the equality \(\mathcal{F} = \BB_X\) holds with the constant ultrametric \(d\),
\[
d(x_i, x_j) = c, \quad c > 0,
\]
for all distinct \(x_i\), \(x_j \in X\).
\end{example}

\textbf{Acknowledgments.} This research was partially supported by State Fund for Fundamental Research of Ukraine (Project F75/28173 ``Investigation of the asymptotic and graphic structures of given classes of metric spaces'').

\end{document}